\documentclass[12pt]{amsart}
\usepackage{fullpage}
\usepackage{graphicx}
\usepackage{todonotes}
\usepackage{tikz}
\usepackage{pgfplots}
\usetikzlibrary{intersections, pgfplots.fillbetween}

\newtheorem{theorem}{Theorem}[section]

\newtheorem{lemma}[theorem]{Lemma}
\newtheorem{proposition}[theorem]{Proposition}

\newtheorem{remark}[theorem]{Remark}

\numberwithin{equation}{section}

\newcommand\NN {{\mathbb N}}

\newcommand\RR {{\mathbb R}}

\newcommand\ZZ {{\mathbb Z}}

\newcommand\cB{{\mathcal{B}  }}

\newcommand\cD{{\mathcal{D}  }}
\newcommand\cE{{\mathcal{E}  }}

\newcommand\cG{{\mathcal{G}  }}

\newcommand\cI{{\mathcal{I}  }}
\newcommand\cJ{{\mathcal{J}  }}
\newcommand\cK{{\mathcal{K}  }}
\newcommand\cL{{\mathcal{L}  }}

\newcommand\cP{{\mathcal{P}  }}
\newcommand\cQ{{\mathcal{Q}  }}
\newcommand\cR{{\mathcal{R}  }}
\newcommand\cS{{\mathcal{S}  }}

\begin{document}



\title{On the measure of products from the middle-third Cantor set}

\author{Luca Marchese}

\address{Dipartimento di Matematica, Universit\`a di Bologna, Piazza di Porta San Donato 5, 40126, Bologna, Italia}

\email{luca.marchese4@unibo.it}

\begin{abstract}
We prove upper and lower bounds for the Lebesgue measure of the set of products $xy$ with $x$ and $y$ in the middle-third Cantor set. Our method is inspired by Athreya, Reznick and Tyson, but a different subdivision of the Cantor set provides a more rapidly converging approximation formula.
\end{abstract}


\maketitle


\section{Introduction}

The middle-third Cantor set is the well-known set $\cK\subset[0,1]$ of points of the form 
$$
x=\sum_{k=1}^\infty\frac{\alpha_k}{3^k},
\quad\text{ where }\quad
\alpha_k\in\{0,2\}
\quad\text{ for }\quad
k\geq1.
$$
Define $P:\RR^2\to\RR$ by $P(x,y):=xy$ and consider $P(\cK\times\cK)$, that is the set of products $xy$ with $x,y\in\cK$, which is a closed set because $P$ is continuous and $\cK$ is compact. 
Denote $\cL$ the Lebesgue measure of $\RR$. The main result of this paper is 
Theorem~\ref{TheoremMainTheorem} below.

\begin{theorem}
\label{TheoremMainTheorem}
We have
$$
\bigg|
\cL\big(P(\cK\times\cK)\big)
-
\frac{91782451}{113374080}
\bigg|
\leq
\frac{1}{10^6}.
$$
\end{theorem}

Previous estimates of $\cL\big(P(\cK\times\cK)\big)$ appear in~\cite{AthreyaReznikTyson} and~\cite{GuJiangXiZhao}, see \S~\ref{SectionApproximationFormula} below for more details. 
Other arithmetic operations with $x,y\in\cK$ are considered in~\cite{AthreyaReznikTyson}, which describes the structure of quotients $y/x$ with $x\not=0$ and proves that $[0,1]$ is covered by products $x^2y$, so that in particular any element of $[0,1]$ is the product of 3 factors in $\cK$. In~\cite{JiangLiWangZhao} it is proved that sums $x_1^2+x_2^2+x_3^2+x_4^2$ with $x_i\in\cK$ for $i=1,2,3,4$ cover $[0,4]$, which was conjectured in~\cite{AthreyaReznikTyson}. In~\cite{JiangXi} is described a general condition on maps $f:\RR^2\to\RR$ such that $f(\cK\times\cK)$ has non-empty interior, where such condition is obviously satisfied by the arithmetic operation mentioned above. For the image under affine maps, and in particular for $S(x,y):=x+y$, a much larger class of Cantor sets and other fractals have been studied. Specific \emph{gap conditions} guarantee that the image is an interval. A first use of gap conditions appears in~\cite{Hall} and a more recent application in~\cite{CabrelliHareMolter}. A gap condition is used for the product of Cantor sets in~\cite{Takahashi}, and for Lipschitz perturbations of $S(\cdot,\cdot)$ 
in Theorem~1.12 in~\cite{ArtigianiMarcheseUlcigrai}. 
Similar ideas are used also in 
\S~\ref{SectionSubdivisionOutsideDiagonal} and \S~\ref{SectionSubdivisionAlongDiagonal} of this paper, inspired by~\cite{AthreyaReznikTyson}. 
Other techniques appear in~\cite{MoreiraYoccoz} and~\cite{SchmelingShmerkin}.

\subsection{Approximation formula}
\label{SectionApproximationFormula}

For $E\subset\RR$ and $c>0$ set $c\cdot E:=\{cx:x\in E\}$. 
The level lines of $P(\cdot,\cdot)$ have \emph{low curvature} in $[2/3,1]^2$. 
For this reason (see \S~\ref{SectionSubdivisionOutsideDiagonal}) we consider the right half $\cR:=\cK\cap[2/3,1]$ of $\cK$. From  
$
\cK=\{0\}\cup\bigcup_{k=0}^\infty 3^{-k}\cdot\cR
$ 
we get 
$$
P(\cK\times\cK)=\{0\}\cup\bigcup_{k=0}^\infty 3^{-k}\cdot P(\cR\times\cR).
$$
The union above is disjoint, because 
$
P(\cR\times\cR)\subset[4/9,1]
$ 
and $1/3<4/9$. Therefore
\begin{equation}
\label{EquationMeasureCantorAndRightHalfCantor}
\cL\big(P(\cK\times\cK)\big)
=
\sum_{k=0}^\infty3^{-k}\cL\big(P(\cR\times\cR)\big)
=
\frac{3}{2}\cL\big(P(\cR\times\cR)\big).
\end{equation}

A \emph{subdivision} for $\cR$ is a nested family of compact sets 
$
\cD_1\supset\dots\supset\cD_n\supset\cD_{n+1}\dots
$ 
such that $\cR=\bigcap_{n=1}^\infty\cD_n$, where any $\cD_{n+1}$ is obtained removing from $\cD_n$ some open intervals (finitely of countably many). A natural subdivision for $\cR$ is 
$\cR_n:=\cK_n\cap[2/3,1]$, where $\cK_n$ is the standard subdivision for $\cK$, 
obtained setting $\cK_0:=[0,1]$ and iteratively  
$$
\cK_n:=
\cK_{n-1}\setminus\bigcup_{j=1}^{3^{n-1}}
\bigg(\frac{3(j-1)+1}{3^n},\frac{3(j-1)+2}{3^n}\bigg)
\quad\text{ for }\quad
n\geq 1.
$$
Theorem~1 in~\cite{AthreyaReznikTyson} gives 
$
17/21<\cL\big(P(\cK\times\cK)\big)<5/6
$. 
These bounds follow from \eqref{EquationMeasureCantorAndRightHalfCantor}, from 
$\cL\big(P(\cR_0\times\cR_0)\big)=5/9$ and from Remark~8 in~\cite{AthreyaReznikTyson}, which gives 
\begin{equation}
\label{EquationFormulaAthreyaReznikTyson}
0\leq
\cL\big(P(\cR_n\times\cR_n)\big)-\cL\big(P(\cR\times\cR)\big)
\leq 
\frac{1}{63}\bigg(\frac{2}{9}\bigg)^n.
\end{equation}

Running a computer program, the authors of \cite{AthreyaReznikTyson} can evaluate 
\eqref{EquationFormulaAthreyaReznikTyson} for $n=11$, which gives 
$
\cL\big(P(\cK\times\cK)\big)=0,80955358\pm10^{-8}
$. 
In~\cite{GuJiangXiZhao} is considered a parameter family of Cantor sets  $\cK^{(\lambda)}$, 
with $\cK^{(\lambda)}=\cK$ for $\lambda=1/3$. A similar approach gives an extension of \eqref{EquationFormulaAthreyaReznikTyson} to any parameter $\lambda$, which provides a sequence of continuous functions of 
$\lambda$ converging uniformly to 
$
\lambda\mapsto\cL\big(P(\cK^{(\lambda)}\times\cK^{(\lambda)})\big)
$, 
so that the latter is continuous too. With a computer program, the authors of~\cite{GuJiangXiZhao} obtain the same first 5 digits in the evaluation of $\cL\big(P(\cK\times\cK)\big)$ by~\cite{AthreyaReznikTyson}. Theorem~\ref{TheoremMainTheorem} confirms such first 5 digits. 
The main tool of this paper is the \emph{fast subdivision} $(\cD_n)_{n\geq1}$ for $\cR$ introduced in 
\S~\ref{SectionDefinitionFastSubdivision}, which gives the rapidly converging approximation formula in Proposition~\ref{PropositionApproximationFastSubdivision} below. 
Figure~\ref{FigureSubdivision} represents one step in such subdivision.

\begin{proposition}
\label{PropositionApproximationFastSubdivision}
Let $(\cD_n)_{n\geq0}$ be the subdivision in \eqref{EquationDefinitionFastSubdivision}. 
For any $n\geq0$ we have
$$
0\leq
\cL\big(P(\cD_n\times\cD_n)\big)-\cL\big(P(\cR\times\cR)\big)
\leq
\frac{1}{63}\bigg(\frac{1}{36}\bigg)^n.
$$
\end{proposition} 

\subsection{Proof of main Theorem~\ref{TheoremMainTheorem}} 

Theorem~\ref{TheoremMainTheorem} follows directly applying \eqref{EquationMeasureCantorAndRightHalfCantor}, 
Proposition~\ref{PropositionApproximationFastSubdivision} with $n=3$, and the next 
Proposition~\ref{PropositionBoundsSecondFastSubdivision}, which is proved in \S~\ref{SectionBoundsSecondSubdivision} below. For the error observe that 
$$
\frac{3}{2}\Big(\frac{1}{64\cdot9^6}+\frac{1}{63\cdot36^3}\Big)
=
\frac{11}{7\cdot16\cdot3^{11}}
<
\frac{1}{10^6}.
$$

\begin{proposition}
\label{PropositionBoundsSecondFastSubdivision}
We have 
$$
\bigg|
\cL\big(P(\cD_3\times\cD_3)\big)-\frac{91782451}{170061120}
\bigg|
\leq
\frac{1}{64\cdot9^6}.
$$
\end{proposition}

\subsection*{Structure of this paper}

In \S~\ref{SectionFastSubdivisionProofApproximationFormula} we introduce the fast subdivision 
$(\cD_n)_{n\geq1}$. The proof of Proposition~\ref{PropositionApproximationFastSubdivision} is completed in \S~\ref{SectionEndProofProofApproximationFormula}. 
In \S~\ref{SectionBoundsSecondSubdivision} we prove Proposition~\ref{PropositionBoundsSecondFastSubdivision}.

\section{The fast subdivision: proof of Proposition~\ref{PropositionApproximationFastSubdivision}}
\label{SectionFastSubdivisionProofApproximationFormula}

\subsection{Definition of the fast subdivision}
\label{SectionDefinitionFastSubdivision}

Consider $\cB\subset(0,1)$ below, which is the disjoint union of countably many open intervals. 
$$
\cB:=
\Bigg(\bigcup_{k\geq1}\bigg(\frac{1}{3^{k+1}},\frac{2}{3^{k+1}}\bigg)\Bigg)
\cup
\bigg(\frac{1}{3},\frac{2}{3}\bigg)
\cup
\Bigg(\bigcup_{k\geq1}\bigg(\frac{3^{k+1}-2}{3^{k+1}},\frac{3^{k+1}-1}{3^{k+1}}\bigg)\Bigg). 
$$
For a closed interval $I=[a,b]$ with $b>a$ let $A_I:[0,1]\to I$ be the unique affine orientation preserving bijection between $[0,1]$ and $I$, that is $A_I(x):=a+(b-a)x$. The closed set $[0,1]\setminus \cB$ is the union of countably many intervals, together with the points $\{0\}$ and $\{1\}$. 
For a set $\cS\subset[0,1]$ which is the union of countably many points and countably many intervals, define 
$$
\cE(\cS):=\{I:\text{ $I$ is a connected component of $\cS$ with non-empty interior}\}.
$$ 
that is the family of intervals of $\cS$. 
For instance $\cE\big([2/3,1]\big)=\{[2/3,1]\}$, and $\cE(\cB)$ denotes the set of open intervals composing $\cB$. As another example, $\big\{\{0\},\{1\}\big\}$ is the set of connected components of $[0,1]\setminus \cB$ that don't belong to
$\cE\big([0,1]\setminus \cB\big)$. Define the \emph{fast subdivision} $(\cD_n)_{n\geq0}$ of $\cR$ setting $\cD_0:=[2/3,1]$ and iteratively 
\begin{equation}
\label{EquationDefinitionFastSubdivision}
\cD_n:=\cD_{n-1}\setminus\bigcup_{I\in\cE(\cD_{n-1})}A_I(\cB)
\quad\text{ for }\quad
n\geq1.
\end{equation} 

\begin{lemma}
\label{LemmaFastSubdivisionSubdividesCantor}
\eqref{EquationDefinitionFastSubdivision} defines a subdivision for $\cR$.
\end{lemma}

\begin{proof}
It is clear that $\cD_n\subset\cD_{n-1}$ for any $n\geq1$. Easy recursive arguments show that $\cD_n$ is closed for any $n\geq0$ and that for any $I\in\cE(\cD_n)$ there exist $m\in\NN$ and $k\geq1$ such that 
$I=[m/3^k,(m+1)/3^k]$. In remains to show $\cR=\bigcap_{n\geq0}\cD_n$. 
For $x\not\in\bigcap_{n\geq0}\cD_n$, consider $n$ with $x\not\in\cD_{n+1}$ and $x\in I=[m/3^k,(m+1)/3^k]$ for some $I\in\cE(\cD_n)$. The definition of $\cB$ implies that there exists $j\geq1$ such that 
$$
\frac{m}{3^k}+\frac{3a+1}{3^{k+j}}<x<\frac{m}{3^k}+\frac{3a+2}{3^{k+j}}
\quad\text{for either }a=0\text{ or }a=3^{j-1}-1,
$$
which implies $x\not\in\cR$. Conversely, if $x\not\in\cR$, then $x\in J$ for an open interval $J$ of the form
\begin{equation}
\label{EquationLemmaFastSubdivisionSubdividesCantor}
J=\bigg(\frac{3a+1}{3^m},\frac{3a+2}{3^m}\bigg)
\quad\text{ for integers }
m\geq2\text{ and }a\geq2\cdot3^{m-2},
\end{equation}
where $m$ is minimal with such property. We prove $x\not\in\bigcap_{n\geq0}\cD_n$ showing by induction on $m$ that there exists $n\geq0$ and an interval $I\in\cE(\cD_n)$ with $J\in\cE\big(A_I(\cB)\big)$, that is $J$ is a gap removed at the $(n+1)$-th step of the fast subdivision. This statement is easily verified for $m=2,3$. We can also assume $J\not\in\cE\big(A_{[2/3,1]}(\cB)\big)$, that is $J$ is not removed at the first step, otherwise the statement is clearly true. For $m\geq4$ consider integers $b_2,\dots,b_{m-1}$ with 
$$
\bigg(\frac{2}{3},1\bigg)
\supset
\bigg(\frac{b_2}{3^2},\frac{1+b_2}{3^2}\bigg)
\supset\dots\supset
\bigg(\frac{b_{m-1}}{3^{m-1}},\frac{1+b_{m-1}}{3^{m-1}}\bigg)
\supset J.
$$ 
The minimality property of $m$ implies $[b_k]\in\{[0],[2]\}$ for $k=2,\dots,m-1$, where $[j]$ denotes the class in $\ZZ/3\ZZ$ of $j\in\ZZ$. If $[b_{m-1}]=[0]$ let $r\geq2$ be minimal with $[b_k]=[0]$ for 
$k=r+1,\dots,m-1$. This minimality property of $r$ implies that 
$$
\widetilde{J}:=
\bigg(\frac{b_r-1}{3^r},\frac{b_r}{3^r}\bigg)
$$
satisfies \eqref{EquationLemmaFastSubdivisionSubdividesCantor}. 
Then we have $n\geq0$ and $\widetilde{I}\in\cE(\cD_n)$ with 
$
\widetilde{J}\in\cE\big(A_{\widetilde{I}}(\cB)\big)
$ 
by inductive assumption, that is $\widetilde{J}$ is a gap removed at the $(n+1)$-step of the fast subdivision. In this case the statement follows because the last condition implies 
$$
I:=\bigg(\frac{b_r}{3^r},\frac{1+b_r}{3^r}\bigg)\in\cE(\cD_{n+1})
\quad\text{ and }\quad
J\in\cE\big(A_I(\cB)\big).
$$
If $[b_{m-1}]=[2]$ consider $r\geq2$ minimal with $[b_k]=[2]$ for 
$k=r+1,\dots,m-1$ and apply the analogous argument to 
$
\widetilde{J}:=\big((1+b_r)/3^r,(2+b_r)/3^r\big)
$. 
\end{proof}

The next \S~\ref{SectionEndProofProofApproximationFormula} requires 
Lemma~\ref{LemmaSumSquaresIntervals} below. Denote $|I|:=\cL(I)$ the length of intervals $I$.  

\begin{lemma}
\label{LemmaSumSquaresIntervals}
For any $n\geq0$ we have 
$$
\sum_{I\in\cE(\cD_n)}|I|^2=\frac{1}{9}\Big(\frac{1}{36}\Big)^n.
$$
\end{lemma}

\begin{proof}
Fix $n\geq1$ and $J\in\cE(\cD_{n-1})$. The intervals of $J\cap\cD_n$ are the images under $A_J$ of the intervals of $[0,1]\setminus\cB$. Since $A_J$ is affine with $dA_J(t)/dt=|J|$ for any $t\in[0,1]$, we have
$$
\sum_{I\in\cE(J\cap\cD_n)}|I|^2
=
|J|^2\sum_{\widetilde{I}\in\cE([0,1]\setminus\cB)}\big|\widetilde{I}\big|^2
=
|J|^2\cdot2\sum_{k=2}^\infty\Big(\frac{1}{3^k}\Big)^2
=
\frac{|J|^2}{36}.
$$
Therefore for any $n\geq1$ we have
$$
\sum_{I\in\cE(\cD_n)}|I|^2
=
\sum_{J\in\cE(\cD_{n-1})}
\Big(\sum_{\widetilde{I}\in\cE(J\cap\cD_n)}
\big|\widetilde{I}\big|^2\Big)
=
\frac{1}{36}\sum_{J\in\cE(\cD_{n-1})}|J|^2
=
\bigg|\Big[\frac{2}{3},1\Big]\bigg|^2
\cdot
\Big(\frac{1}{36}\Big)^n.
$$
\end{proof}

\subsection{Subdivisions outside the diagonal}
\label{SectionSubdivisionOutsideDiagonal}

Consider any interval $I\subset[2/3,1]$ and write $I=[a,a+3t]$ with $2/3\leq a<a+3t\leq1$. Then set 
\begin{equation}
\label{EquationNotationMiddleThirdGap}
\ddot{I}:=[a,a+t]\cup[a+2t,a+3t].
\end{equation}
Consider intervals $I,J\subset[2/3,1]$ such that $I=[a,a+3t]$ and $J=[b,b+3t]$, and assume that 
$I\cap J=\emptyset$. According to Lemma~11 in~\cite{AthreyaReznikTyson} we have 
\begin{equation}
\label{EquationSubdivisionOutsideDiagonal}
P(\ddot{I}\times\ddot{J})=P(I\times J).
\end{equation}
The elementary proof follows computing the extremal values of $P(\cdot,\cdot)$ over the four connected components of $\ddot{I}\times\ddot{J}$ and checking that the images overlap. 

\begin{lemma}
\label{LemmaContinuousImageCompactIntersection}
Let $f:\RR^2\to\RR$ be a continuous function and $(C_n)_{n\geq0}$ be a sequence of compact sets of $\RR^2$ with $C_{n+1}\subset C_n$ for any $n\geq0$. Then 
$$
f(C)=\bigcap_{n\geq0}f(C_n)
\quad\text{ where }\quad
C:=\bigcap_{n\geq0}C_n.
$$
\end{lemma}

\begin{proof}
Since $f(C)\subset f(C_k)$ for any $k\geq0$, then
$
f(C)\subset \bigcap_{k\geq0} f(C_k)
$. 
Fix $y\in \bigcap_{n\geq0}f(C_n)$. For any $k\geq0$ there exists $x_k\in C_k$ with $y=f(x_k)$. Since $x_k\in C_0$ for any $k$, modulo subsequences we have $x_k\to x$ for some $x\in C_0$. 
Continuity implies $y=f(x)$. We must have $x\in \bigcap_{n\geq0}C_n$, indeed otherwise there exists $N\geq1$ and a neighborhood of $U$ of $x$ with $x_k\in U$ and $U\cap C_k=\emptyset$ for any $n\geq N$, which is absurd. Therefore 
$
\bigcap_{k\geq0} f(C_k)\subset f(C)
$. 
\end{proof}

Recall that $\cK\subset[0,1]$ denotes the middle-third Cantor set. 

\begin{lemma}
\label{LemmaSubdivisionOutsideDiagonal}
Consider intervals $I,J\subset[2/3,1]$ such that $I=[a,a+3t]$ and $J=[b,b+3t]$, and assume that 
$I\cap J=\emptyset$. Then we have
$$
P\big(A_I(\cK)\times A_J(\cK)\big)=P(I\times J).
$$
\end{lemma}

\begin{proof}
Set $\cI_0:=I$ and $\cJ_0:=J$ and for $n\geq1$ define inductively the compact sets 
$$
\cI_n:=\bigcup_{E\in\cE(\cI_{n-1})}\ddot{E}
\quad\text{ and }\quad
\cJ_n:=\bigcup_{F\in\cE(\cJ_{n-1})}\ddot{F},
$$
where $\cI_n\subset\cI_{n-1}$ and $\cJ_n\subset\cJ_{n-1}$. We have  
$
A_I(\cK)=\bigcap_{n=0}^\infty\cI_n
$ 
and 
$
A_J(\cK)=\bigcap_{n=0}^\infty\cJ_n
$. 
For any $n\geq0$, any pair of intervals $(E,F)$ with $E\in\cE(\cI_n)$ and $F\in\cE(\cJ_n)$ satisfies the same assumption as the pair $(I,J)$ in the statement. 
Then \eqref{EquationSubdivisionOutsideDiagonal} implies 
$$
P(\cI_{n+1}\times\cJ_{n+1})
=
\bigcup_{E\in\cE(\cI_n),F\in\cE(\cJ_n)}P(\ddot{E}\times\ddot{F})
=
\bigcup_{E\in\cE(\cI_n),F\in\cE(\cJ_n)}P(E\times F)
=
P(\cI_n\times\cJ_n).
$$
This implies $P(\cI_n\times\cJ_n)=P(I\times J)$ for any $n\geq0$. 
Lemma~\ref{LemmaContinuousImageCompactIntersection} gives
$$
P\big(A_I(\cK)\times A_J(\cK)\big)
=
\bigcap_{n\geq0}P(\cI_n\times\cJ_n)
=
P(I\times J).
$$
\end{proof}

\subsection{Subdivision along the diagonal} 
\label{SectionSubdivisionAlongDiagonal}

For an interval $I\subset[2/3,1]$ write $I=[a,a+3t]$ with $2/3\leq a<a+3t\leq1$. 
In the notation of \eqref{EquationNotationMiddleThirdGap} we have 
\begin{equation}
\label{EquationSubdivisionDiagonal}
P(I\times I)\setminus P(\ddot{I}\times\ddot{I})
=
\big((a+2t)^2-t^2,(a+2t)^2\big),
\end{equation}
because $(a+t)^2>a(a+2t)$ and $(a+t)(a+3t)=(a+2t)^2-t^2<(a+2t)^2$. 
For $I$ as above consider the map $A_I:[0,1]\to I$. 
We have $A_I(t)=A_I(0)+|I|t$, where $|I|:=\cL(I)$. Set 
$$
\cD_I:=I\setminus A_I(\cB).
$$
For $0\leq x<y\leq1$ set 
$
I(x,y):=\big[A_I(x),A_I(y)\big]
$. 
For $k\geq0$ define the intervals 
\begin{align*}
&
\cP_{(I,k,-)}:=
P\big(I(2/3^{k+1},1/3^k)\times I(0,1/3^{k+1})\big)
\\
&
\cP_{(I,k,+)}:=
P\big(I(1-1/3^{k+1},1)\times I(1-1/3^k,1-2/3^{k+1})\big),
\end{align*}
where  
$
\cP_{(I,0,-)}=\cP_{(I,0,+)}=P\big(I(2/3,1)\times I(0,1/3)\big)
$.

\begin{lemma}
\label{LemmaSubdivisionOutsideDiagonal(Local)}
For any $I\subset[2/3,1]$ and any $k\geq0$ we have
$$
\cP_{(I,k,-)},\cP_{(I,k,+)}\subset P\big(A_I(\cK)\times A_I(\cK)\big).
$$
\end{lemma}

\begin{proof}
We have 
$
A_{I(0,1/3)}(\cK)\subset A_I(\cK)
$ 
and 
$
A_{I(2/3,1)}(\cK)\subset A_I(\cK)
$. 
The statement follows for $\cP_{(I,0,-)}=\cP_{(I,0,+)}$ because Lemma~\ref{LemmaSubdivisionOutsideDiagonal} gives 
$$
P\big(I(2/3,1)\times I(0,1/3)\big)
=
P\big(A_{I(2/3,1)}(\cK)\times A_{I(0,1/3)}(\cK)\big).
$$
The same argument applies to $\cP_{(I,k,-)}$ and $\cP_{(I,k,+)}$ for $k\geq1$.
\end{proof}

For $k\geq0$ set $I_{(k,-)}:=I(0,1/3^k)$ and $I_{(k,+)}:=I(1-1/3^k,1)$, where $I_{(0,-)}=I_{(0,+)}=I$. 
Recall \eqref{EquationSubdivisionDiagonal} and for $k\geq0$ define the open intervals
\begin{align*}
&
\cG_{(I,k,-)}:=
P(I_{(k,-)}\times I_{(k,-)})\setminus P(\ddot{I}_{(k,-)}\times\ddot{I}_{(k,-)})
\\
&
\cG_{(I,k,+)}:=
P(I_{(k,+)}\times I_{(k,-)})\setminus P(\ddot{I}_{(k,+)}\times\ddot{I}_{(k,+)}).
\end{align*}
Observe that  
$
\cG_{(I,0,-)}=\cG_{(I,0,+)}=P(I\times I)\setminus P(\ddot{I}\times\ddot{I})
$. 
Then set 
$$
\cG_I:=
\Big(\bigcup_{k=1}^\infty\cG_{(I,k,-)}\Big)
\cup
\cG_{(I,0,-)}
\cup
\Big(\bigcup_{k=1}^\infty\cG_{(I,k,+)}\Big).
$$
Finally for $k\geq1$ denote the elements of $\cE(\cD_I)$ by 
$$
D_{(I,k,-)}:=I(2/3^{k+1},1/3^k)
\quad\text{ and }\quad
D_{(I,k,+)}:=I(1-1/3^k,1-2/3^{k+1})
$$
and define the intervals
$$
\cQ_{_{(I,k,-)}}:=P(D_{(I,k,-)}\times D_{(I,k,-)})
\quad\text{ and }\quad
\cQ_{_{(I,k,+)}}:=P(D_{(I,k,+)}\times D_{(I,k,+)}).
$$

\begin{remark}
\label{RemarkOrderIntervalsLocalSubdivisionDiagonal}
Fix $I$. For simplicity write $A=A_I$. For $x,y,z,t\in\RR$ we have 
\begin{equation}
\label{EquationComparisonExtremalValues}
A(x)A(y)>A(z)A(t)
\quad\Leftrightarrow\quad
zt-xy<\frac{A(0)}{|I|}(x+y-z-t).
\end{equation}
Using \eqref{EquationComparisonExtremalValues}, and observing that $A(0)/|I|\geq 2$ 
for intervals $I\subset[2/3,1]$, it is easy to verify that for for any $k\geq0$ we have
\begin{align*}
&
\sup\cG_{(I,k+1,-)}=\inf \cQ_{(I,k+1,-)}<\inf\cP_{(I,k,-)}
\quad\Leftrightarrow\quad
A(2/3^{k+2})^2<A(0)A(2/3^{k+1})
\\
&
\inf\cP_{(I,k,-)}<\sup \cQ_{(I,k+1,-)}
\quad\Leftrightarrow\quad
A(0)A(2/3^{k+1})<A(1/3^{k+1})^2
\\
&
\sup \cQ_{(I,k+1,-)}<\sup\cP_{(I,k,-)}
\quad\Leftrightarrow\quad
A(1/3^{k+1})^2<A(1/3^k)A(1/3^{k+1})
\\
&
\sup\cP_{(I,k,-)}=\inf\cG_{(I,k,-)}
<
\sup\cG_{(I,k,-)}\leq\sup\cG_{(I,k,+)}=\inf \cQ_{(I,k+1,+)}
\end{align*}
and for any $k\geq1$ we have 
\begin{align*}
&
\inf \cQ_{(I,k,+)}<\inf \cP_{(I,k,+)}
\quad\Leftrightarrow\quad
A(1-1/3^k)^2<A(1-1/3^{k+1})A(1-1/3^k)
\\
&
\inf \cP_{(I,k,+)}
<
\sup \cQ_{(I,k,+)}
\quad\Leftrightarrow\quad
A(1-1/3^{k+1})A(1-1/3^k)<A(1-2/3^{k+1})^2
\\
&
\sup \cQ_{(I,k,+)}
<
\sup \cP_{(I,k,+)}
\quad\Leftrightarrow\quad
A(1-2/3^{k+1})^2<A(1)A(1-2/3^{k+1})
\\
&
\sup \cP_{(I,k,+)}=\inf\cG_{(I,k,+)}<\sup\cG_{(I,k,+)}=\inf \cQ_{(I,k+1,+)}.
\end{align*}
The equalities above follow from \eqref{EquationSubdivisionDiagonal}. 
See Figure~\ref{FigureSubdivision}.
\end{remark}

Remark~\ref{RemarkOrderIntervalsLocalSubdivisionDiagonal} implies that $\cG_I$ is a disjoint union. 
Then \eqref{EquationSubdivisionDiagonal} implies 
\begin{equation}
\label{EquationLocalPictureMeasureOfGaps}
\cL(\cG_I)
=
\cL(\cG_{(I,0,-)})
+
\sum_{k=1}^\infty\cL(\cG_{(I,k,-)})+\cL(\cG_{(I,k,-)})
=
\frac{|I^2|}{3^2}+2\sum_{k=1}^\infty\frac{|I|^2}{3^{2k+2}}
=
\frac{5}{36}|I|^2.
\end{equation}

\begin{proposition}
\label{PropositionLocalPictureGaps}
Fix an interval $I\subset[2/3,1]$. We have
\begin{equation}
\label{Equation(1)PropositionLocalPictureGaps}
P(I\times I)\setminus P(\cD_I\times\cD_I)=\cG_I.
\end{equation}
\end{proposition}

\begin{proof}
Write $E^2:=E\times E$ for sets $E\subset[0,1]$. The set 
$
\cD_I^2\cap\{(x,y)\in\RR^2:y\leq x\}
$ 
is contained in the union over $k\geq0$ of the sets 
$$
\cD_{(I,k+1,-)}^2\cup\cD_{(I,k+1,+)}^2\cup
\Big(I(2/3^{k+1},1/3^k)\times I_{(k+1,-)}\Big)
\cup
\Big(I_{(k+1,+)}\times I(1-1/3^k,1-2/3^{k+1})\Big)
$$
This implies 
\begin{equation}
\label{EquationPropositionLocalPictureGaps}
P(\cD_I^2)
\subset
\Big(\bigcup_{k\geq0}\cP_{(I,k,-)}\cup\cP_{(I,k,+)}\Big)
\cup
\Big(\bigcup_{k\geq1}\cQ_{(I,k,-)}\cup \cQ_{(I,k,+)}\Big).
\end{equation}
On the other hand $A_I(\cK)\subset\cD_I$ by Lemma~\ref{LemmaFastSubdivisionSubdividesCantor}. 
Hence Lemma~\ref{LemmaSubdivisionOutsideDiagonal(Local)} implies that $P(\cD_I^2)$ contains 
$\cP_{(I,k,-)}$ and $\cP_{(I,k,+)}$ for $k\geq0$. Obviously $P(\cD_I^2)$ also contains 
$\cQ_{(I,k,-)}$ and $\cQ_{(I,k,+)}$ for any $k\geq1$. 
Therefore the inclusion in \eqref{EquationPropositionLocalPictureGaps} is indeed an equality between sets. The intervals in $\cG_I$ fill the gaps in $P(\cD_I^2)$ 
by Remark~\ref{RemarkOrderIntervalsLocalSubdivisionDiagonal}. 
This proves \eqref{Equation(1)PropositionLocalPictureGaps}.
\end{proof}

\subsection{End of the proof of Proposition~\ref{PropositionApproximationFastSubdivision}}
\label{SectionEndProofProofApproximationFormula}

Fix $n\geq1$ and consider two intervals $J_1,J_2$ in $\cE(\cD_n)$ with $J_1\not=J_2$. Without loss of generality assume $\sup J_2<\inf J_1$. 
Let $m\leq n-1$ be maximal such that there exists $I\in\cE(\cD_m)$ with $J_i\subset I$ for $i=1,2$. Maximality implies that $J_1$ and $J_2$ are included into different connected components of $\cD_I$. If $J_1\subset D_{(I,k,-)}$ for some $k\geq1$, then 
Lemma~\ref{LemmaSubdivisionOutsideDiagonal(Local)} implies 
$$
P(J_1\times J_2)
\subset
P\Big(
I(2/3^{k+1},1/3^k)\times I(0,1/3^{k+1})
\Big)
=
\cP_{(I,k,-)}
\subset
P(\cR\times\cR),
$$
where we recall that $A_I(\cK)\subset\cR$ by Lemma~\ref{LemmaFastSubdivisionSubdividesCantor}. Otherwise there exists $l\geq0$ with 
$$
J_1\times J_2\subset I(1-1/3^{l+1},1)\times I(1-1/3^l,1-2/3^{l+1})
$$
and we get again 
$
P(J_1\times J_2)\subset\cP_{(I,l,+)}\subset P(\cR\times\cR)
$ 
by Lemma~\ref{LemmaSubdivisionOutsideDiagonal(Local)}. Both inclusions cannot be derived directly from \eqref{EquationSubdivisionOutsideDiagonal} because, a priori, $J_1$ and $J_2$ have different sizes. Thus  
$$
P\Big((\cD_n\times\cD_n)\setminus\bigcup_{I\in\cE(\cD_n)}I\times I\Big)
\subset 
P(\cR\times\cR)
\quad\text{ for any }\quad 
n\geq0.
$$
Since $P(\cR\times\cR)\subset P(\cD_n\times\cD_n)$ for any $n\geq0$, then \eqref{EquationDefinitionFastSubdivision} and 
\eqref{Equation(1)PropositionLocalPictureGaps} give
\begin{align*}
P(\cD_n\times\cD_n)\setminus P(\cD_{n+1}\times\cD_{n+1})
&
\subset
\bigcup_{I\in\cE(D_n)}
P(I\times I)\setminus P(\cD_{n+1}\times\cD_{n+1})
\\
&
\subset
\bigcup_{I\in\cE(D_n)}
P(I\times I)\setminus P(\cD_I\times\cD_I)
=
\bigcup_{I\in\cE(D_n)}\cG_I.
\end{align*}
Therefore \eqref{EquationLocalPictureMeasureOfGaps} and 
Lemma~\ref{LemmaSumSquaresIntervals} give
$$
\cL\big(P(\cD_n\times\cD_n)\big)
-
\cL\big(P(\cD_{n+1}\times\cD_{n+1})\big)
\leq
\sum_{I\in\cE(D_n)}\cL(\cG_I)
=
\frac{5}{36}\sum_{I\in\cE(D_n)}|I|^2
=
\frac{5}{9\cdot36^{n+1}}.
$$
For any $n\geq0$ and $m\geq n+1$ a telescopic argument gives
$$
0
\leq 
\cL\big(P(\cD_n\times\cD_n)\big)
-
\cL\big(P(\cD_m\times\cD_m)\big)
\leq
\sum_{k=n+1}^{m}
\frac{5/9}{36^k}
\leq
\sum_{k=n+1}^\infty
\frac{5/9}{36^k}
=
\frac{1}{63\cdot36^n}.
$$
We have  
$
\cL\big(P(\cD_m\times\cD_m)\big)\to \cL\big(P(\cR\times\cR)\big)
$ 
as $m\to\infty$. 
Proposition~\ref{PropositionApproximationFastSubdivision} is proved.

\begin{figure}[h]
\begin{center}

\begin{tikzpicture}[scale=6]


\draw[-,very thin] (0,0) -- (1,0) -- (1,1) -- (0,1) -- (0,0);

\draw[-,very thin] (0,0) -- (1,1);


\fill[black!22!white] (2/3,0) rectangle (1,1/3);
\fill[black!22!white] (0,2/3) rectangle (1/3,1);

\fill[black!22!white] (2/9,0) rectangle (1/3,1/9);
\fill[black!22!white] (0,2/9) rectangle (1/9,1);

\fill[black!22!white] (8/9,2/3) rectangle (1,7/9);
\fill[black!22!white] (2/3,8/9) rectangle (7/9,1);

\fill[black!22!white] (2/27,0) rectangle (1/9,1/27);
\fill[black!22!white] (0,2/27) rectangle (1/27,1/3);

\fill[black!22!white] (26/27,8/9) rectangle (1,25/27);
\fill[black!22!white] (8/9,26/27) rectangle (25/27,1);



\draw[-,very thin] (1/3,0) -- (1/3,1);

\draw[-,very thin] (2/3,0) -- (2/3,1);

\draw[-,very thin] (0,1/3) -- (1,1/3);

\draw[-,very thin] (0,2/3) -- (1,2/3);


\draw[-,very thin] (1/9,0) -- (1/9,1);

\draw[-,very thin] (2/9,0) -- (2/9,1);

\draw[-,very thin] (7/9,0) -- (7/9,1);

\draw[-,very thin] (8/9,0) -- (8/9,1);

\draw[-,very thin] (0,1/9) -- (1,1/9);

\draw[-,very thin] (0,2/9) -- (1,2/9);

\draw[-,very thin] (0,7/9) -- (1,7/9);

\draw[-,very thin] (0,8/9) -- (1,8/9);


\draw[-,very thin] (1/27,0) -- (1/27,1);

\draw[-,very thin] (2/27,0) -- (2/27,1);

\draw[-,very thin] (7/27,0) -- (7/27,1);

\draw[-,very thin] (8/27,0) -- (8/27,1);

\draw[-,very thin] (19/27,0) -- (19/27,1);

\draw[-,very thin] (20/27,0) -- (20/27,1);

\draw[-,very thin] (25/27,0) -- (25/27,1);

\draw[-,very thin] (26/27,0) -- (26/27,1);

\draw[-,very thin] (0,1/27) -- (1,1/27);

\draw[-,very thin] (0,2/27) -- (1,2/27);

\draw[-,very thin] (0,7/27) -- (1,7/27);

\draw[-,very thin] (0,8/27) -- (1,8/27);

\draw[-,very thin] (0,19/27) -- (1,19/27);

\draw[-,very thin] (0,20/27) -- (1,20/27);

\draw[-,very thin] (0,25/27) -- (1,25/27);

\draw[-,very thin] (0,26/27) -- (1,26/27);


\fill[white] (0,1/3) rectangle (1/3,2/3);

\fill[white] (2/3,1/3) rectangle (1,2/3);

\fill[white] (1/3,2/3) rectangle (2/3,1);

\fill[white] (1/3,0) rectangle (2/3,1/3);


\fill[white] (0,1/9) rectangle (1/9,2/9);

\fill[white] (2/9,1/9) rectangle (1/3,2/9);

\fill[white] (1/9,2/9) rectangle (2/9,1/3);

\fill[white] (1/9,0) rectangle (2/9,1/9);









\fill[white] (2/3,7/9) rectangle (7/9,8/9);

\fill[white] (8/9,7/9) rectangle (1,8/9);

\fill[white] (7/9,8/9) rectangle (8/9,1);

\fill[white] (7/9,2/3) rectangle (8/9,7/9);


\fill[black] (2/9,2/9) rectangle (1/3,1/3);

\fill[black] (2/3,2/9) rectangle (7/9,1/3);

\fill[black] (2/9,6/9) rectangle (1/3,7/9);

\fill[black] (2/3,6/9) rectangle (7/9,7/9);


\fill[black] (2/27,2/27) rectangle (1/9,1/9);

\fill[black] (8/9,2/27) rectangle (25/27,1/9);

\fill[black] (2/27,8/9) rectangle (1/9,25/27);

\fill[black] (8/9,8/9) rectangle (25/27,25/27);

\fill[black] (2/9,2/27) rectangle (1/3,1/9);

\fill[black] (2/3,2/27) rectangle (7/9,1/9);

\fill[black] (8/9,2/9) rectangle (25/27,1/3);

\fill[black] (8/9,6/9) rectangle (25/27,7/9);

\fill[black] (2/9,8/9) rectangle (1/3,25/27);

\fill[black] (2/3,8/9) rectangle (7/9,25/27);

\fill[black] (2/27,2/9) rectangle (1/9,1/3);

\fill[black] (2/27,6/9) rectangle (1/9,7/9);


\fill[black] (2/81,2/81) rectangle (1/27,1/27);

\fill[black] (26/27,2/81) rectangle (79/81,1/27);

\fill[black] (2/81,26/27) rectangle (1/27,79/81);

\fill[black] (26/27,26/27) rectangle (79/81,79/81);

\fill[black] (2/9,2/81) rectangle (1/3,1/27);

\fill[black] (2/3,2/81) rectangle (7/9,1/27);

\fill[black] (26/27,2/9) rectangle (79/81,1/3);

\fill[black] (26/27,6/9) rectangle (79/81,7/9);

\fill[black] (2/9,26/27) rectangle (1/3,79/81);

\fill[black] (2/3,26/27) rectangle (7/9,79/81);

\fill[black] (2/81,2/9) rectangle (1/27,1/3);

\fill[black] (2/81,6/9) rectangle (1/27,7/9);

\fill[black] (2/27,2/81) rectangle (1/9,1/27);

\fill[black] (8/9,2/81) rectangle (25/27,1/27);

\fill[black] (2/27,26/27) rectangle (1/9,79/81);

\fill[black] (8/9,26/27) rectangle (25/27,79/81);

\fill[black] (2/81,2/27) rectangle (1/27,1/9);

\fill[black] (2/81,8/9) rectangle (1/27,25/27);

\fill[black] (26/27,2/27) rectangle (79/81,1/9);

\fill[black] (26/27,8/9) rectangle (79/81,25/27);


\draw[-,very thick] (2/81,0) -- (1/27,0);
\draw[-,very thick] (2/27,0) -- (1/9,0);
\draw[-,very thick] (2/9,0) -- (1/3,0);
\draw[-,very thick] (2/3,0) -- (7/9,0);
\draw[-,very thick] (8/9,0) -- (25/27,0);
\draw[-,very thick] (26/27,0) -- (79/81,0);

\draw[-,very thick] (2/81,1) -- (1/27,1);
\draw[-,very thick] (2/27,1) -- (1/9,1);
\draw[-,very thick] (2/9,1) -- (1/3,1);
\draw[-,very thick] (2/3,1) -- (7/9,1);
\draw[-,very thick] (8/9,1) -- (25/27,1);
\draw[-,very thick] (26/27,1) -- (79/81,1);

\draw[-,very thick] (0,2/81) -- (0,1/27);
\draw[-,very thick] (0,2/27) -- (0,1/9);
\draw[-,very thick] (0,2/9) -- (0,1/3);
\draw[-,very thick] (0,2/3) -- (0,7/9);
\draw[-,very thick] (0,8/9) -- (0,25/27);
\draw[-,very thick] (0,26/27) -- (0,79/81);

\draw[-,very thick] (1,2/81) -- (1,1/27);
\draw[-,very thick] (1,2/27) -- (1,1/9);
\draw[-,very thick] (1,2/9) -- (1,1/3);
\draw[-,very thick] (1,2/3) -- (1,7/9);
\draw[-,very thick] (1,8/9) -- (1,25/27);
\draw[-,very thick] (1,26/27) -- (1,79/81);



\node at (0.45,0.6) {$\cP_{(I,0,\pm)}$};

\draw[->,very thin,dotted] (1.2,0.4) -- (0.94,0.4);
\node at (1.32,0.4) {$\cG_{(I,0,\pm)}$};

\draw[->,very thin,dotted] (-0.2,0.37) -- (0.03,0.37);
\node at (-0.32,0.37) {$\cP_{(I,1,-)}$};

\draw[->,very thin,dotted] (-0.3,0.28) -- (0.21,0.28);
\node at (-0.42,0.28) {$\cQ_{(I,1,-)}$};

\draw[->,very thin,dotted] (-0.3,0.72) -- (0.65,0.72);
\node at (-0.42,0.72) {$\cQ_{(I,1,+)}$};

\draw[->,very thin,dotted] (-0.2,0.48) -- (-0.01,0.48);
\node at (-0.32,0.48) {$\cG_{(I,1,-)}$};

\draw[->,very thin,dotted] (1.2,0.63) -- (0.97,0.63);
\node at (1.32,0.63) {$\cP_{(I,1,+)}$};

\draw[->,very thin,dotted] (1.2,0.78) -- (1.01,0.78);
\node at (1.32,0.78) {$\cG_{(I,1,+)}$};


\clip(0,0) rectangle (1,1);


\draw[domain=0:1, variable=\x] plot ({\x}, {3*( (53/81)^2*((\x+1)/3)^-1)-1});
\draw[name path=G, domain=0:1, variable=\x] plot ({\x}, {3*( (2/3)*(52/81)*((\x+1)/3)^-1)-1});
\draw[name path=H, domain=0:1, variable=\x] plot ({\x}, {3*( (17/27)*(53/81)*((\x+1)/3)^-1)-1});


\draw[domain=0:1, variable=\x] plot ({\x}, {3*( 0.398*((\x+1)/3)^-1)-1});
\draw[name path=E, domain=0:1, variable=\x] plot ({\x}, {3*( 0.394*((\x+1)/3)^-1)-1});
\draw[name path=F, domain=0:1, variable=\x] plot ({\x}, {3*( (5/9)*(17/27)*((\x+1)/3)^-1)-1});

\draw[domain=0:1, variable=\x] plot ({\x}, {3*( (5/9)^2*((\x+1)/3)^-1)-1});
\draw[name path=A, domain=0:1, variable=\x] plot ({\x}, {3*( (2/3)*(4/9)*((\x+1)/3)^-1)-1});
\draw[name path=B, domain=0:1, variable=\x] plot ({\x}, {3*( (1/3)*(5/9)*((\x+1)/3)^-1)-1});


\draw[domain=0:1, variable=\x] plot ({\x}, {3*( 0.167*((\x+1)/3)^-1)-1});
\draw[name path=C, domain=0:1, variable=\x] plot ({\x}, {3*( 0.163*((\x+1)/3)^-1)-1});
\draw[name path=D, domain=0:1, variable=\x] plot ({\x}, {3*( (1/3)*(11/27)*((\x+1)/3)^-1)-1});

\draw[domain=0:1, variable=\x] plot ({\x}, {3*( (29/81)^2*((\x+1)/3)^-1)-1});
\draw[name path=I, domain=0:1, variable=\x] plot ({\x}, {3*( (10/27)*(28/81)*((\x+1)/3)^-1)-1});
\draw[name path=L, domain=0:1, variable=\x] plot ({\x}, {3*( (1/3)*(29/81)*((\x+1)/3)^-1)-1});

\tikzfillbetween[of=A and B]{gray, opacity=0.1};
\tikzfillbetween[of=C and D]{gray, opacity=0.1};
\tikzfillbetween[of=E and F]{gray, opacity=0.1};
\tikzfillbetween[of=G and H]{gray, opacity=0.1};
\tikzfillbetween[of=I and L]{gray, opacity=0.1};





\end{tikzpicture}

\end{center}
\caption{By Lemma~\ref{LemmaSubdivisionOutsideDiagonal}, each dark grey square has the same $P(\cdot,\cdot)$-image as its intersection with $\cD_I\times\cD_I$, represented in black. Such images are the intervals $\cP_{(I,k,\pm)}$, $k\geq0$. White regions of hyperbolas not intersecting $\cD_I\times\cD_I$ correspond to the gaps $\cG_{(I,k,\pm)}$, $k\geq0$. The intervals $\cQ_{(I,k,\pm)}$, $k\geq1$ are the images of the black squares $J\times J$ along the diagonal, where $J\in\cE(\cD_I)$. The superposition between such black squares and the light grey regions generates covered gaps at the next subdivision.}
\label{FigureSubdivision}
\end{figure}

\section{Proof of Proposition~\ref{PropositionBoundsSecondFastSubdivision}}
\label{SectionBoundsSecondSubdivision}

\subsection{Covered gaps}
\label{SectionCoveredGaps}

Use the notation of \S~\ref{SectionSubdivisionAlongDiagonal}. 
Fix $m\geq0$ and an interval $J\in\cE(\cD_m)$. A gap $G\in\cE(\cG_J)$ is \emph{covered} if $G\subset P(\cR\times\cR)$. Such gap $G$ doesn't give negative contribution to the measure of $P(\cR\times\cR)$. In the following, for $N\geq1$, we replace 
\eqref{EquationLocalPictureMeasureOfGaps} by 
\begin{equation}
\label{EquationTruncatedSumGaps}
\cL
\Big(
\bigcup_{0\leq l\leq N-1}\cG_{(J,l,+)}
\cup
\bigcup_{1\leq l<\infty}\cG_{(J,l,-)}
\Big)
=
\sum_{0\leq l\leq L-1}\frac{|J|^2}{9^{l+1}}
+
\sum_{1\leq l\leq \infty}\frac{|J|^2}{9^{l+1}}
=
\frac{|J|^2}{8}\Big(\frac{10}{9}-\frac{1}{9^N}\Big).
\end{equation}

Fix $n\geq0$ and $I\in\cE(\cD_n)$. 
For $k\geq1$ consider $E:=D_{(I,k,-)}$, which is an element of $\cE(\cD_I)$, 
and $\cQ_{(I,k,-)}=P(E\times E)$. We have 
$
\inf\cP_{(I,k-1,-)}<\sup P(E\times E)
$ 
by Remark~\ref{RemarkOrderIntervalsLocalSubdivisionDiagonal}.
See also Figure~\ref{FigureSubdivision}. 
Recalling Lemma~\ref{LemmaSubdivisionOutsideDiagonal(Local)}, for $l>>1$ we get 
\begin{equation}
\label{EquationGapsEntirelyCovered}
\cG_{(E,l,+)}\subset\cP_{(I,k-1,-)}\subset P(\cR\times\cR).
\end{equation}
Hence the subdivision of $E$ (at step $n+2$, after the subdivision of $I$, at step $n+1$) generates a tail of covered gaps $\cG_{(E,l,+)}$. We also have 
$
\inf P(E\times E)=\sup\cG_{(i,k,-)}<\inf\cP_{(I,k,-)}
$, 
which follows again from Remark~\ref{RemarkOrderIntervalsLocalSubdivisionDiagonal}. 
Hence for $l>>1$ we have 
\begin{equation}
\label{EquationGapsEntirelyUncovered}
\cG_{(E,l,-)}\cap P\big((\cD_I\times\cD_I)\setminus(E\times E)\big)=\emptyset.
\end{equation} 
Lemma~\ref{LemmaCoveredGaps} below implies that we always have either \eqref{EquationGapsEntirelyCovered} or \eqref{EquationGapsEntirelyUncovered}, that is we don't have gaps just partially covered by 
$
P\big((I\cap\cR)\times(I\cap\cR)\big)
$. 
Moreover covered gaps are determined by an arithmetic condition. The same discussion applies to $F:=D_{(I,k,+)}$ with $k\geq1$, indeed 
$
\sup\cG_{(I,k-1,+)}=\inf\cQ_{(I,k,+)}
$ 
and 
$
\inf\cP_{(I,k,+)}<\sup\cQ_{(I,k,+)}
$ 
by Remark~\ref{RemarkOrderIntervalsLocalSubdivisionDiagonal}. Thus 
$$
\cG_{(F,l,+)}\subset\cP_{(I,k,+)}
\quad\text{ and }\quad
\cG_{(F,l,-)}\cap P\big((\cD_I\times\cD_I)\setminus(F\times F)\big)=\emptyset
\quad\text{ for }\quad
l>>1.
$$

\subsection{Arithmetic condition for covered gaps}
\label{SectionArithmeticConditionCoveredGaps}

Fix $n\geq0$ and $I\in\cE(\cD_n)$ as in \S~\ref{SectionCoveredGaps}. 
For $k\geq1$ consider $E:=D_{(I,k,-)}$, that is $E=I(2/3^{k+1},1/3^k)$. 
Consider the affine maps $A_I:[0,1]\to I$ and $A_E:[0,1]\to E$. For $x\in\RR$ we have
$$
A_I(x)=
A_E\bigg(\frac{A_I(x)-A_E(0)}{|E|}\bigg)=
A_E\bigg(\frac{3^{k+1}}{|I|}\Big(x-\frac{2}{3^{k+1}}\Big)|I|\bigg)=
A_E(3^{k+1}x-2).
$$
We have 
$
\inf\cP_{(I,k-1,-)}=A_I(0)A_I(2/3^k)=A_E(-2)A_E(4)
$. 
From the definition of $\cB$ in \S~\ref{SectionDefinitionFastSubdivision}, it is clear that for $l\geq0$ we have  
\begin{equation}
\label{EquationSectionArithmeticConditionCoveredGaps(endpoints)}
\inf\cG_{(E,l,+)}=A_E(1)A_E\Big(1-\frac{2}{3^{l+1}}\Big)
\quad\text{ and }\quad
\sup\cG_{(E,l,+)}=A_E\Big(1-\frac{1}{3^{l+1}}\Big)^2.
\end{equation}
Finally  
$
A_E(0)/|E|=3^{k+1}A_I(0)|I|+2
$. 
Hence \eqref{EquationComparisonExtremalValues} gives 
\begin{align*}
&
A_E(-2)A_E(4)
<
A_E(1)
A_E\Big(1-\frac{2}{3^{l+1}}\Big)
\Leftrightarrow
-8-1+\frac{2}{3^{l+1}}
<
-\frac{A_E(0)}{|E|}\frac{2}{3^{l+1}}
\Leftrightarrow
\\
&
\hspace{1 cm}
-9+\frac{2}{3^{l+1}}
<
-\Big(3^{k+1}\frac{A_I(0)}{|I|}+2\Big)\frac{2}{3^{l+1}}
\Leftrightarrow
3^{l+2}
>
2\cdot3^k\frac{A_I(0)}{|I|}+2.
\\
&
A_E(-2)A_E(4)
<
A_E\Big(1-\frac{1}{3^{l+1}}\Big)^2
\Leftrightarrow
-8-\Big(\frac{3^{l+1}-1}{3^{l+1}}\Big)^2
<
-\frac{A_E(0)}{|E|}\frac{2}{3^{l+1}}
\Leftrightarrow
\\
&
\hspace{1 cm}
3^{l+3}+\frac{1}{3^{l+1}}
>
2\cdot3^{k+1}\frac{A_I(0)}{|I|}+6
\Leftrightarrow
3^{l+2}>2\cdot3^k\frac{A_I(0)}{|I|}+2,
\end{align*}
where the last equivalence holds because $A_I(0)/|I|$ is integer (this can be seen by induction on $n$) and therefore we always have  
$
3^{l+3}\not=2\cdot3^{k+1}A_I(0)/|I|+6
$. 
We get
\begin{equation}
\label{EquationAritmeticConditionLEFT}
\inf\cP_{(I,k-1,-)}<\inf\cG_{(E,l,+)}
\Leftrightarrow
3^{l+2}>2\cdot3^k\frac{A_I(0)}{|I|}+2
\Leftrightarrow
\inf\cP_{(I,k-1,-)}<\sup\cG_{(E,l,+)}.
\end{equation}

Similarly, for $n\geq0$ and $I\in\cE(\cD_n)$ as in \S~\ref{SectionCoveredGaps}, and for $k\geq1$ consider $F:=D_{(I,k,+)}$, that is 
$
F=I(1-1/3^k,1-2/3^{k+1})
$. 
The maps $A_I:[0,1]\to I$ and $A_F:[0,1]\to E$ satisfy 
$$
A_I(x)=
A_F\bigg(\frac{3^{k+1}}{|I|}\Big(x-1+\frac{1}{3^k}\Big)|I|\bigg)=
A_F\big(3^{k+1}(x-1)+3\big)
\quad\text{ for }\quad 
x\in\RR.
$$
We have 
$
\inf\cP_{(I,k,+)}=A_I(1-1/3^{k+1})A_I(1-1/3^k)=A_F(2)A_F(0)
$.
For $l\geq0$ the expression of $\inf\cG_{(F,l,+)}$ and $\sup\cG_{(F,l,+)}$ in terms of $A_F$ is as in 
\eqref{EquationSectionArithmeticConditionCoveredGaps(endpoints)}. 
Thus \eqref{EquationComparisonExtremalValues} gives 
\begin{align*}
&
A_F(2)A_F(0)
<
A_F(1)
A_F\Big(1-\frac{2}{3^{l+1}}\Big)
\Leftrightarrow
\frac{2}{3^{l+1}}-1
<
-\frac{A_F(0)}{|F|}\frac{2}{3^{l+1}}
\Leftrightarrow
\\
&
\hspace{1 cm}
\frac{2}{3^{l+1}}-1
<
-\bigg(3^{k+1}\Big(\frac{A_I(0)}{|I|}+1\Big)-3\bigg)\frac{2}{3^{l+1}}
\Leftrightarrow
3^{l+1}
>
2\cdot3^{k+1}\Big(\frac{A_I(0)}{|I|}+1\Big)-4.
\\
&
A_F(2)A_F(0)
<
A_F\Big(1-\frac{1}{3^{l+1}}\Big)^2
\Leftrightarrow
-\Big(\frac{3^{l+1}-1}{3^{l+1}}\Big)^2
<
-\frac{A_F(0)}{|F|}\frac{2}{3^{l+1}}
\Leftrightarrow
\\
&
\hspace{1 cm}
3^{l+1}+\frac{1}{3^{l+1}}
>
2\cdot3^{k+1}\Big(\frac{A_I(0)}{|I|}+1\Big)-4
\Leftrightarrow
3^{l+1}
>
2\cdot3^{k+1}\Big(\frac{A_I(0)}{|I|}+1\Big)-4,
\end{align*}
where again the last equivalence holds because $A_I(0)/|I|$ is integer and therefore we always have  
$
3^{l+1}\not=2\cdot3^{k+1}\big(A_I(0)/|I|+1\big)-4
$. 
We get
\begin{equation}
\label{EquationAritmeticConditionRIGHT}
\inf\cP_{(I,k,+)}<\inf\cG_{(F,l,+)}
\Leftrightarrow
3^{l+1}>2\cdot3^{k+1}\Big(\frac{A_I(0)}{|I|}+1\Big)-4
\Leftrightarrow
\inf\cP_{(I,k,-)}<\sup\cG_{(F,l,+)}.
\end{equation}

In Lemma~\ref{LemmaCoveredGaps} below, fix $n\geq0$, $I\in\cE(\cD_n)$ and $k\geq1$. 


\begin{lemma}
\label{LemmaCoveredGaps}
We have 
$
L(k):=\min\{l\geq0:\text{ \eqref{EquationAritmeticConditionLEFT} holds }\}\geq k+2n
$. 
For $E:=D_{(I,k,-)}$ and for any $l\geq L(k)$ we have 
$
\cG_{(E,l,+)}\subset P(\cR\times\cR)
$ 
and on the other hand 
$$
\Big(
\bigcup_{0\leq l\leq L(k)-1}\cG_{(E,l,+)}
\cup
\bigcup_{l=1}^\infty\cG_{(E,l,-)}
\Big)
\cap
P\big((I\cap\cD_{n+2})\times(I\cap\cD_{n+2})\big)
=\emptyset.
$$ 
Moreover    
$
R(k):=\min\{l\geq0:\text{ \eqref{EquationAritmeticConditionRIGHT} holds }\}\geq k+2n
$. 
For $F:=D_{(I,k,+)}$ and for any $l\geq R(k)$ we have 
$
\cG_{(F,l,+)}\subset P(\cR\times\cR)
$ 
and on the other hand 
$$
\Big(
\bigcup_{0\leq l\leq R(k)-1}\cG_{(F,l,+)}
\cup
\bigcup_{l=1}^\infty\cG_{(F,l,-)}
\Big)
\cap
P\big((I\cap\cD_{n+2})\times(I\cap\cD_{n+2})\big)
=\emptyset.
$$
\end{lemma}

\begin{proof}
We have $A_I(0)/|I|\geq2\cdot9^n$ for any $n\geq0$ and any $I\in\cE(\cD_n)$. This can be easily proved by induction on $n$, observing that $A_I(0)/|I|=2$ for $I=[2/3,1]=\cD_0$. Therefore \eqref{EquationAritmeticConditionLEFT} and \eqref{EquationAritmeticConditionRIGHT} give 
$L(k)\geq k+2n$ and $R(k)\geq k+2n$. Moreover \eqref{EquationAritmeticConditionLEFT} implies 
$$
\cG_{(E,L(k),+)}\subset P(\cR\times\cR)
\quad\text{ and }\quad
\cG_{(E,L(k)-1,+)}\cap P\big((\cD_I\times\cD_I)\setminus(E\times E)\big)=\emptyset,
$$
that is no gap $\cG_{(E,l,+)}$ is partially covered. The order between gaps in Remark~\ref{RemarkOrderIntervalsLocalSubdivisionDiagonal} gives 
$$
\Big(
\bigcup_{0\leq l\leq L(k)-1}\cG_{(E,l,+)}
\cup
\bigcup_{0\leq l<\infty}\cG_{(E,l,-)}
\Big)
\cap 
P\big((\cD_I\times\cD_I)\setminus(E\times E)\big)=\emptyset.
$$
The first statement follows. The second statement follows by a similar argument.
\end{proof}

\subsection{End of the proof of Proposition~\ref{PropositionBoundsSecondFastSubdivision}}

Set $I:=[2/3,1]$. We have $A_I(0)/|I|=2$. For any $n\geq1$ define 
$$
\mu_n:=\cL\big(P(I\times I)\setminus P(\cD_n\times\cD_n)\big).
$$
According to \eqref{EquationDefinitionFastSubdivision} we have $\cD_1=\cD_{[2/3,1]}$, thus 
\eqref{Equation(1)PropositionLocalPictureGaps} and 
\eqref{EquationLocalPictureMeasureOfGaps} give
$$
\mu_1=\cL(\cG_{[2/3,1]})
=
\frac{5}{36}\big|[2/3,1]\big|^2
=
\frac{5}{4\cdot81}.
$$

For $k\geq1$ consider $E_k:=I(2/3^{k+1},1/3^k)$ and $F_k:=I(1-1/3^k,1-2/3^{k+1})$. 
We have 
$
|E_k|=|F_k|=|I|/3^{k+1}=1/3^{k+2}
$. 
Since $A_I(0)/|I|=2$, it is easy to see that in Lemma~\ref{LemmaCoveredGaps} we have $L(k)=k$ and $R(k)=k+2$. 
Thus Lemma~\ref{LemmaCoveredGaps} and \eqref{EquationTruncatedSumGaps} give
\begin{align*}
\mu_2-\mu_1
&
=
\sum_{k=1}^\infty\frac{|E_k|^2}{8}\Big(\frac{10}{9}-\frac{1}{9^k}\Big)
+
\sum_{k=1}^\infty\frac{|F_k|^2}{8}\Big(\frac{10}{9}-\frac{1}{9^{k+2}}\Big)
=
\frac{859}{9^4\cdot5\cdot64}.
\end{align*}

Fix $k\geq1$ and $F=F_k$. For $m\geq1$ let $D(m,\pm):=D_{(F,m,\pm)}$ be the intervals arising from the subdivision of $F$ at the $2$-nd step $F$. 
At the $3$-rd step any $D=D(m,\pm)$ generates gaps $\cG_{(D,l,\pm)}$ with $l\geq0$. 
We have $A_F(0)/|F|=3^{k+2}-3$. Thus 
\eqref{EquationAritmeticConditionLEFT} and \eqref{EquationAritmeticConditionRIGHT} give
\begin{align*}
&
\cG_{(D(m,-),l,+))}\subset \cP_{(F,m-1,-)}
\quad\Leftrightarrow\quad
3^{l+2}>2\cdot3^m(3^{k+2}-3)+2
\quad\Leftrightarrow\quad
l\geq m+k+1.
\\
&
\cG_{(D(m,+),l,+))}\subset\cP_{(F,m,+)}
\quad\Leftrightarrow\quad
3^{l+1}>2\cdot3^{m+1}(3^{k+2}-2)-4
\quad\Leftrightarrow\quad
l\geq m+k+3.
\end{align*}
Recall Remark~\ref{RemarkOrderIntervalsLocalSubdivisionDiagonal}, and that in our notation 
$
P\big(D(m,\pm)\times D(m,\pm)\big)=\cQ_{(F,m,\pm)}
$. 
For any $m\geq k+3$ we have 
$
\cQ_{(F,m,+)}\subset\cP_{(I,k,+)}
$, 
indeed $R(k)=k+2$ and thus
$$
\inf\cG_{(F,k+1,+)}<\inf\cP_{(I,k,+)}<\inf\cG_{(F,k+2,+)}<\inf\cQ_{(F,k+3,+)}\leq\inf\cQ_{(F,m,+)}.
$$ 
Moreover 
$
\cQ_{(F,k+2,+)}\cap\cP_{(I,k,+)}=\emptyset
$, 
indeed using \eqref{EquationComparisonExtremalValues} as in \S~\ref{SectionArithmeticConditionCoveredGaps} we get 
\begin{align*}
&
\sup\cQ_{(F,k+2,+)}<\inf\cP_{(I,k,+)}
\Leftrightarrow
A_F(2)A_F(0)
>
A_F\Big(1-\frac{2}{3^{k+3}}\Big)^2
\Leftrightarrow
\\
&
\hspace{ 1 cm }
\Big(\frac{3^{k+3}-2}{3^{k+3}}\Big)^2
<
\frac{A_F(0)}{|F|}\frac{4}{3^{k+3}}
\Leftrightarrow 
8+\frac{4}{3^{k+3}}<3^{k+2},
\end{align*}
which is true for any $k\geq1$. Finally  
$
\sup\cG_{(I,k-1,+)}=\inf\cQ_{(I,k,+)}<\inf\cQ_{(F,m,-)}
$ 
for $m\geq1$. Hence 
$$
\Big(
\bigcup_{m=1}^\infty\cQ_{(F,m,-)}
\cup
\bigcup_{m=1}^{k+2}\cQ_{(F,m,+)}
\Big)
\cap
\Big(\cP_{(I,k-1,+)}\cup\cP_{(I,k,+)}\Big)
=
\emptyset.
$$ 
Let $\beta_k$ be the total measure of non-covered gaps generated by intervals 
$
D(m,\pm)\in\cE(\cD_F)
$, 
where $F=F_k$ and $m\geq1$. The discussion above, Lemma~\ref{LemmaCoveredGaps} and 
\eqref{EquationTruncatedSumGaps} imply 
\begin{align*}
\beta_k
&
=
\sum_{m=1}^\infty\frac{|D(m,-)|^2}{8}\Big(\frac{10}{9}-\frac{1}{9^{m+k+1}}\Big)
+
\sum_{m=1}^{k+2}\frac{|D(m,+)|^2}{8}\Big(\frac{10}{9}-\frac{1}{9^{m+k+3}}\Big)
\\
&
=
\frac{1}{8}\sum_{m=1}^\infty\frac{1}{9^{k+m+3}}\Big(\frac{10}{9}-\frac{1}{9^{m+k+1}}\Big)
+
\frac{1}{8}\sum_{m=1}^{k+2}\frac{1}{9^{k+m+3}}\Big(\frac{10}{9}-\frac{1}{9^{m+k+3}}\Big)
\\
&
=
\frac{1}{64}
\Big(\frac{20}{9^{k+4}}-\frac{91/5}{9^{2k+6}}+\frac{1/10}{9^{4k+10}}\Big).
\end{align*}

Now fix $k\geq1$ and $E=E_k$. For $m\geq1$ let $D(m,\pm):=D_{(E,m,\pm)}$ be the intervals arising from the subdivision of $E$ at the $2$-nd step. 
At the $3$-rd step any $D=D(m,\pm)$ generates gaps $\cG_{(D,l,\pm)}$ with $l\geq0$. 
We have $A_E(0)/|E|=2\cdot3^{k+1}+2$. Thus 
\eqref{EquationAritmeticConditionLEFT} and \eqref{EquationAritmeticConditionRIGHT} give
\begin{align*}
&
\cG_{(D(m,-),l,+))}\subset \cP_{(F,m-1,-)}
\quad\Leftrightarrow\quad
3^{l+2}>2\cdot3^m(2\cdot3^{k+1}+2)+2
\quad\Leftrightarrow\quad
l\geq m+k+1.
\\
&
\cG_{(D(m,+),l,+))}\subset\cP_{(F,m,+)}
\quad\Leftrightarrow\quad
3^{l+1}>2\cdot3^{m+1}(2\cdot3^{k+1}+3)-4
\quad\Leftrightarrow\quad
l\geq m+k+3.
\end{align*}
We have 
$
\inf\cG_{(E,k-1,+)}<\inf\cP_{(I,k-1,-)}<\inf\cG_{(E,k,+)}<\inf\cQ_{(E,k+1,+)}
$ 
because $L(k)=k$. As above, Remark~\ref{RemarkOrderIntervalsLocalSubdivisionDiagonal} gives 
$
\cQ_{(E,m,+)}\subset\cP_{(I,k-1,-)}
$ 
for any $m\geq k+1$, and  
$$
\Big(
\bigcup_{m=1}^\infty\cQ_{(E,m,-)}
\cup
\bigcup_{m=1}^{k-1}\cQ_{(E,m,+)}
\Big)
\cap
\Big(\cP_{(I,k,-)}\cup\cP_{(I,k-1,+)}\Big)
=
\emptyset.
$$ 
It remains to determine whether 
$
\cQ_{(E,k,+)}\cap\cP_{(I,k,-)}=\emptyset
$ 
or not. 
From \S~\ref{SectionArithmeticConditionCoveredGaps} and \eqref{EquationComparisonExtremalValues} we get
\begin{align*}
&
\sup\cQ_{(E,k,+)}
<
\inf\cP_{(I,k-1,-)}
\Leftrightarrow
A_E(-2)A_E(4)
>
A_E\Big(1-\frac{2}{3^{k+1}}\Big)^2
\Leftrightarrow
\\
&
\hspace{ 1 cm }
\Big(\frac{3^{k+1}-2}{3^{k+1}}\Big)^2+8
<
\frac{A_E(0)}{|E|}\frac{4}{3^{k+1}}
\Leftrightarrow
3^{k+1}+\frac{4}{3^{k+1}}<12,
\end{align*}
which is true for $k=1$ and false for $k\geq2$. Recall that in our notation $E=E_k$ and 
$
\cE(\cD_E)=\{D(m,\pm)=D_{(E,m,\pm)}:m\geq1\}
$. 
In particular 
$
D(k,+)=D_{(E_k,k,+)}\subset E_k
$. 
Let $\alpha_k$ be the total measure of non covered gaps generated by intervals 
$
D(m,\pm)\in\cE(\cD_E)\setminus\{D(k,+)\}
$. 
Lemma~\ref{LemmaCoveredGaps} and \eqref{EquationTruncatedSumGaps} give
\begin{align*}
\alpha_k
&
=
\sum_{m=1}^\infty\frac{|D(m,-)|^2}{8}\Big(\frac{10}{9}-\frac{1}{9^{m+k+1}}\Big)
+
\sum_{m=1}^{k-1}\frac{|D(m,+)|^2}{8}\Big(\frac{10}{9}-\frac{1}{9^{m+k+3}}\Big)
\\
&
=
\beta_k-
\sum_{m=k}^{k+2}\frac{|D(m,+)|^2}{8}\Big(\frac{10}{9}-\frac{1}{9^{m+k+3}}\Big)
=
\beta_k-\frac{1}{8}\sum_{j=3,4,5}\frac{1}{9^{2k+j}}\Big(\frac{10}{9}-\frac{1}{9^{2k+j}}\Big).
\end{align*}

Let $\gamma_k$ be the total measure of non covered gaps generated only by $D(k,+)$. 
For $k=1$ we have 
$
\cQ_{(E,1,+)}\cap\cP_{(I,0,-)}=\emptyset
$. 
Then applying \eqref{EquationTruncatedSumGaps} with $N=k+m+3$ for $m=k=1$ we get 
$$
\gamma_1=
\frac{|D(1,+)|^2}{8}\Big(\frac{10}{9}-\frac{1}{9^{2+3}}\Big)=
\frac{5}{4\cdot9^6}-\frac{1}{8\cdot9^{10}}.
$$
Conversely
$
\cQ_{(E,k,+)}\cap\cP_{(I,k-1,-)}\not=\emptyset
$ 
for $k\geq2$. Thus $\cP_{(I,k-1,-)}$ cover some gaps arising from $D(k,+)$ and we can only give an upper bound on $\gamma_k$ using \eqref{EquationLocalPictureMeasureOfGaps}. We obtain 
$$
0
\leq
\mu_3
-
\Big(\mu_2+\gamma_1+\sum_{k\geq1}(\beta_k+\alpha_k)\Big)
\leq
\sum_{k\geq2}\gamma_k
\leq
\sum_{k\geq2}\cL(\cG_{D(k,+)})
=
\sum_{k\geq2}\frac{5}{36}\frac{1}{9^{2k+3}}
=
\frac{1}{64\cdot9^6}.
$$

We have 
$$
\sum_{k\geq1}\beta_k+\alpha_k
=
\frac{1}{8}
\sum_{k=1}^\infty
\frac{5}{9^{k+4}}
-
\frac{910}{9^{2k+6}}
-
\frac{91/20}{9^{2k+6}}
+
\frac{1/40}{9^{4k+10}}
+
\frac{6643}{9^{4k+10}}
=
\frac{157}{32\cdot9^6}-\epsilon
$$ 
where 
$$
\epsilon:=
\frac{1}{8}\sum_{k\geq1}
\frac{91/20}{9^{2k+6}}
-
\frac{1/40}{9^{4k+10}}
-
\frac{6643}{9^{4k+10}}
\in
\Big(0,\frac{1}{128\cdot9^6}\Big).
$$
Set 
$
M:=\mu_2+\gamma_1+157/(32\cdot9^6)
$, 
so that 
$$
\frac{5}{9}-M=
\frac{5}{9}-
\Big(
\frac{5}{4\cdot81}
+
\frac{859}{9^4\cdot5\cdot64}
+
\frac{5}{4\cdot9^6}-\frac{1}{8\cdot9^{10}}
+
\frac{157}{32\cdot9^6}
\Big)
=
\frac{91782451}{170061120}+\frac{1}{8\cdot9^{10}}.
$$ 
We have 
$
\cL(P(\cD_3\times\cD_3))=5/9-\mu_3
$. 
Therefore 
$$
\frac{-1}{64\cdot9^6}
<
\frac{-1}{64\cdot9^6}
+
\frac{1}{8\cdot9^{10}}
+
\epsilon
\leq 
\cL\big(P(\cD_3\times\cD_3)\big)-\frac{91782451}{170061120}
\leq
\frac{1}{8\cdot9^{10}}+\epsilon
<
\frac{1}{64\cdot9^6}.
$$ 
Proposition~\ref{PropositionBoundsSecondFastSubdivision} is proved.


\begin{thebibliography}{0000}

\bibitem{ArtigianiMarcheseUlcigrai}
M. Artigiani, L. Marchese, C. Ulcigrai:
\emph{Persistent Hall rays for Lagrange Spectra at cusps of Riemann surfaces}, 
Ergodic Theory and Dynamical Systems, Vol. 40, Issue 8, 2020, 2017-2072.

\bibitem{AthreyaReznikTyson}
J. Athreya, B. Reznick, J. Tyson: 
\emph{Cantor set arithmetic}. 
American Mathematical Monthly, 126:1 (2019), 4-17.

\bibitem{CabrelliHareMolter}
C. A. Cabrelli, K. E. Hare, U. M. Molter:
\emph{Sums of Cantor sets yielding an interval}.
J. Aust. Math. Soc. Vol. 73, n. 3 (2002), 405-418.

\bibitem{GuJiangXiZhao}
J. Gu, K. Jiang, L. Xi, B. Zhao:
\emph{Multiplication on uniform $\lambda$-Cantor sets}. 
ArXiv:1910.08303.

\bibitem{JiangLiWangZhao}
K. Jiang, W. Li, Z. Wang, B. Zhao:
\emph{On the sum of squares of middle-third Cantor set}.
Journal of Number Theory, Vol. 218 (2021), 209-222.

\bibitem{JiangXi}
K. Jiang, L. Xi:
\emph{Interiors of continuous images of the middle-third Cantor set}.
ArXiv:1809.01880.

\bibitem{Hall}
M. Hall:
\emph{On the sum and products of continued fractions}.
Annals of Mathematics (2), Vol. 48, n. 4 (1947), 966-993. 

\bibitem{MoreiraYoccoz}
C. G. Moreira, and J.-C. Yoccoz:
\emph{Stable intersections of regular Cantor sets with large Hausdorff dimensions}.
Annals of Mathematics, Vol. 154, n.1 (2001), 45-96.

\bibitem{SchmelingShmerkin}
J. Schmeling, P. Shmerkin: (2010). 
\emph{On the dimension of iterated sumsets.} 
In J. Barral, S. Seuret eds. Recent Developments in Fractals and Related Fields. Appl. Numer. Harmon. Anal., Birkh\"auser Boston, 55-72.

\bibitem{Takahashi}
Y. Takahashi:
\emph{Products of two Cantor sets.} 
Nonlinearity, Vol. 30 n. 5 (2017), 2114-2137.


\end{thebibliography}
\end{document}